\documentclass[12pt]{amsart}
\usepackage{amscd,amssymb,amsthm,amsmath,enumerate,amssymb,textcomp,enumerate,longtable}
\usepackage[matrix,arrow]{xy}
\usepackage{url}
\usepackage{color, colortbl}

\usepackage{float}
\restylefloat{table}

\definecolor{Gray3}{gray}{0.8}
\definecolor{Gray1}{gray}{0.5}

\theoremstyle{definition}

\newtheorem*{example*}{Example}

\newtheorem{theorem}[equation]{Theorem}
\newtheorem{lemma}[equation]{Lemma}
\newtheorem{corollary}[equation]{Corollary}

\newtheorem{conjecture}[equation]{Conjecture}
\newtheorem*{conjecture*}{Conjecture}

\newtheorem*{question*}{Question}

\newtheorem*{problem*}{Problem}
\newtheorem*{theorem*}{Theorem}

\newcommand{\PP}{\mathbb{P}}

\theoremstyle{remark}

\newtheorem*{remark*}{Remark}

\makeatletter\@addtoreset{equation}{section} \makeatother

\usepackage{xcolor}

\title{On geometry of Fano threefold hypersurfaces}

\author{Hamid Ahmadinezhad, Ivan Cheltsov, Jihun Park}

\begin{document}

\begin{abstract} 
We prove that a quasi-smooth Fano threefold hypersurface is birationally rigid if and only if it has Fano index one.
\end{abstract}

\address{\emph{Hamid Ahmadinezhad}
\newline
\textnormal{Department of Mathematical Sciences, Loughborough University, Loughborough LE11 3TU, UK.}
\newline
\textnormal{\url{h.ahmadinezhad@lboro.ac.uk}}}

\address{\emph{Ivan Cheltsov}
\newline
\textnormal{School of Mathematics, The University of Edinburgh,  Edinburgh, UK.}
\newline
\textnormal{Laboratory of Algebraic Geometry and its applications, Higher School of Economics, Moscow, Russia.}
\newline
\textnormal{\url{I.Cheltsov@ed.ac.uk}}}

\address{ \emph{Jihun Park}\newline \textnormal{Centre for Geometry and Physics, Institute for Basic Science
\newline \medskip 77 Cheongam-ro, Nam-gu, Pohang, Gyeongbuk, 37673, Korea \newline
Department of Mathematics, POSTECH
\newline
77 Cheongam-ro, Nam-gu, Pohang, Gyeongbuk,  37673, Korea \newline
\texttt{wlog@postech.ac.kr}}}



\maketitle

\section{Introduction}
\label{section:intro}

End points of Minimal Model Program are either Mori fibre spaces or minimal models. 
In dimension three, Mori fibre spaces form three classes: Fano threefolds, del Pezzo fibrations over curves, and conic bundles over surfaces.
They are $\mathbb{Q}$-factorial with at worst terminal singularities, and with relative Picard number one. 
The focus of this article is on Fano threefolds.
They lie in finitely many deformation families (see \cite{Birkar} and \cite{Kawamata}),
and studying birational relations among them, as well as birational maps to other Mori fibre spaces is fundamental, 
as it sheds light to birational classification of rationally connected threefolds in general.

For a Fano threefold $X$ with terminal $\mathbb{Q}$-factorial singularities, let $A$ be a Weil divisor for which $-K_X=\iota_X A$ with maximum $\iota_X\in\mathbb{Z}_{>0}$.
This integer $\iota_X$ is known as the Fano \emph{index} of $X$.
It follows from \cite{ABR} and \cite{Brown-Suzuki} that there are precisely $130$ families of Fano threefolds such that the $\mathbb{Z}$-graded ring
$$
R(X,A)=\bigoplus_{m\geqslant 0}H^0\big(X,-mA\big)
$$ 
has $5$ generators. Denote these generators by  $x,y,z,t,w$, respectively in degrees $a_0,a_1,a_2,a_3,a_4$, and let the algebraic relation amongst them be $f(x,y,z,t,w)=0$ of weighted degree $d$. 
This redefines $X$ as a hypersurface in the weighted projective space $\mathbb{P}(a_0,a_1,a_2,a_3,a_4)$ that is given by 
the quasi-homogeneous equation
$$
f(x,y,z,t,w)=0,
$$
and the Fano index of $X$ is computed via $\iota_X=\sum_{i=0}^{4}a_i-d$.
Among these $130$ families, exactly $95$ have index one and their birational geometry is well-studied. Notably, the following theorem holds.

\begin{theorem}[{\cite{CP,CPR}}] 
\label{theorem:CPR}
All index one quasi-smooth Fano threefold hypersurfaces are birationally rigid.
\end{theorem}

In particular, all index one quasi-smooth Fano threefold hypersurfaces are irrational.
The remaining families of Fano threefold hypersurfaces with $\iota_X\geqslant 2$ are listed in Table~\ref{table-rat},
and their rationality has been studied by several people: 
the irrationality of every smooth cubic threefold (family \textnumero\,96) has been proved by Clemens and Griffiths \cite{Clemens-Griffiths},
the irrationality of all smooth threefolds in family \textnumero\,97 has been proved by Voisin \cite{Voisin},
and the irrationality all smooth threefolds in family \textnumero\,98 has been proved by Grinenko \cite{Grin1,Grin2}.
Recently, the complete classification of families whose general member is irrational was carried out by Okada~\cite{Okada-rational}. 
Whether a general Fano threefold in a given family is rational or not is also indicated in Table~\ref{table-rat}. 

\begin{table}[h]\small
\label{table-rat}
\caption{\footnotesize Fano threefold hypersurfaces of index $\iota_X\geqslant 2$ and their rationality data.}
$$
\renewcommand\arraystretch{1.5}
\begin{array}{|c|c|c|c||c|c|c|c|}
\hline
\text{\textnumero}&X_d\subset\PP(a_0,a_2,a_2,a_3,a_4)&\iota_X&\text{Rational}&\text{\textnumero}&X_d\subset\PP(a_0,a_2,a_2,a_3,a_4)&\iota_X&\text{Rational}\\
\hline
\hline
96&X_{3}\subset\PP(1,1,1,1,1)&2&\text{No}&114&X_6\subset\PP(1,1,2,3,4)&5&\text{Yes}\\
\hline
97 &X_{4}\subset\PP(1,1,1,1,2)&2&\text{No}&115&X_6\subset\PP(1,2,2,3,3)&5&\text{Yes}\\
\hline
98 &X_{6}\subset\PP(1,1,1,2,3)&2&\text{No}&116&X_{10}\subset\PP(1,2,3,4,5)&5&\text{No}\\
\hline
99 &X_{10}\subset\PP(1,1,2,3,5)&2&\text{No}&117&X_{15}\subset\PP(1,3,4,5,7)&5&\text{No}\\
\hline
100 &X_{18}\subset\PP(1,2,3,5,9)&2&\text{No}&118&X_6\subset\PP(1,1,2,3,5)&6&\text{Yes}\\
\hline
101 &X_{22}\subset\PP(1,2,3,7,11)&2&\text{No}&119&X_6\subset\PP(1,2,2,3,5)&7&\text{Yes}\\
\hline
102 &X_{26}\subset\PP(1,2,5,7,13)&2&\text{No}&120&X_6\subset\PP(1,2,3,3,4)&7&\text{Yes}\\
\hline
103 &X_{38}\subset\PP(2,3,5,11,19)&2&\text{No}&121&X_8\subset\PP(1,2,3,4,5)&7&\text{Yes}\\
\hline
104 &X_{2}\subset\PP(1,1,1,1,1)&3&\text{Yes}&122&X_{14}\subset\PP(2,3,4,5,7)&7&\text{No}\\
\hline
105 &X_3\subset\PP(1,1,1,1,2)&3&\text{Yes}&123&X_6\subset\PP(1,2,3,3,5)&8&\text{Yes}\\
\hline
106 &X_4\subset\PP(1,1,1,2,2)&3&\text{Yes}&124&X_{10}\subset\PP(1,2,3,5,7)&8&\text{Yes}\\
\hline
107 &X_6\subset\PP(1,1,2,2,3)&3&\text{No}&125&X_{12}\subset\PP(1,3,4,5,7)&8&\text{Yes}\\
\hline
108 &X_{12}\subset\PP(1,2,3,4,5)&3&\text{No}&126&X_6\subset\PP(1,2,3,4,5)&9&\text{Yes}\\
\hline
109 &X_{15}\subset\PP(1,2,3,5,7)&3&\text{No}&127&X_{12}\subset\PP(2,3,4,5,7)&9&\text{Yes}\\
\hline
110 &X_{21}\subset\PP(1,3,5,7,8)&3&\text{No}&128&X_{12}\subset\PP(1,4,5,6,7)&11&\text{Yes}\\
\hline
111 &X_4\subset\PP(1,1,1,2,3)&4&\text{Yes}&129&X_{10}\subset\PP(2,3,4,5,7)&11&\text{Yes}\\
\hline
112&X_6\subset\PP(1,1,2,3,3)&4&\text{Yes}&130&X_{12}\subset\PP(3,4,5,6,7)&13&\text{Yes}\\
\hline
113&X_4\subset\PP(1,1,2,2,3)&5&\text{Yes}&&&&\\
\hline
\end{array}
$$
\end{table}

We are interested in irrational Fano threefolds, and their birational geometry. 
In this paper, we prove that, contrary to index one case, there are no birationally rigid Fano threefold hypersurfaces with $\iota_X\geqslant 2$.
To be precise, our main result is 

\begin{theorem}
\label{main-theorem} 
Let $X$ be a quasi-smooth Fano threefold in family \textnumero\,$n$ in Table~\ref{table-rat}. Then
\begin{enumerate}[(i)]
\item there exists a birational map to a Mori fibration over a curve or a surface if
$$
n\not\in\{100, 101, 102, 103, 110\} \quad\text{(see Section\,\ref{fibration})},
$$
\item there exists an elementary Sarkisov link from $X$ to another Fano threefold if
$$
n\in\{100, 101, 102, 103, 110\}\quad \text{(see Section\,\ref{Fano-models})}.
$$
\end{enumerate}
In particular, if  $\iota_X\geqslant 2$, then $X$ is not birationally rigid.
\end{theorem}

\begin{corollary} 
\label{corollary:main}
A quasi-smooth Fano threefold hypersurface with terminal singularities is birationally rigid if and only if it has index one.
\end{corollary}

Our results go beyond the proof of birational non-rigidity. 
They also provide evidence for the following conjecture (see \cite[Definition\,1.4]{AO} for the definition of solid Fano varieties). 

\begin{conjecture} 
\label{conjecture:main}
A quasi-smooth Fano hypersurface $X$ with $\iota_X\geqslant 2$ is solid if and only if it belongs to one of the families \textnumero 100, \textnumero 101, \textnumero 102, \textnumero 103, \textnumero 110.
\end{conjecture}

In Section\,\ref{fibration}, we prove one direction of this conjecture, i.e., we show that 
if a quasi-smooth Fano threefold hypersurface $X$ with $\iota_X\geqslant 2$ is solid, 
then it must belong to one of the families \textnumero 100, \textnumero 101, \textnumero 102, \textnumero 103, \textnumero 110.
In Section\,\ref{sketch}, we indicate how to prove the other direction of our Conjecture~\ref{conjecture:main}.
Unfortunately, the technical difficulties in the final step of the indicated proof do not allow us to finish the proof.
We encourage the reader to fill this gap.

\subsection*{Acknowledgement} 
This project was initiated during a Research-in-Teams programme at the Erwin Schr\"odinger Institute for Mathematics and Physics during August 2018. 
The authors would like to express their gratitude to the ESI and its staff for their support and hospitality. The first author has been supported by EPSRC (grant EP/T015896/1), and the third author has been supported by IBS-R003-D1, Institute for Basic Science in Korea.

\section{Non-solid Fano threefolds}
\label{fibration}

Let $X$ be a quasi-smooth Fano threefold in family \textnumero\,$n$ in Table~\ref{table-rat}.
Then $X$ is a hypersurface in the weighted projective space $\mathbb{P}(a_0,a_1,a_2,a_3,a_4)$ of degree $d$. 
Suppose also that $n\not\in\{100,101,102,103,110\}$. Then the Fano index $i_X$ of $X$ satisfies   
$$
a_0a_1<i_X=a_0+a_1+a_2+a_3+a_4-d.
$$
Here, we assume that $a_0\leqslant a_1\leqslant a_2\leqslant a_3\leqslant a_4$ as in Table~\ref{table-rat}.

Let $\psi\colon X\dasharrow\mathbb{P}^1$ be the map given by the projection  $\mathbb{P}(a_0,a_1,a_2,a_3,a_4)\dasharrow\mathbb{P}(a_0,a_1)$.
Then there exists a commutative diagram 
$$
\xymatrix{
&\widetilde{X}\ar@{->}[ld]_{\pi}\ar@{->}[rd]^{\phi}&\\%
X\ar@{-->}[rr]_{\psi}&&\mathbb{P}^1}
$$
where $\widetilde{X}$ is a smooth threefold, $\pi$ is a birational map, and $\phi$ is a surjective morphism.
Let~$S$ be a sufficiently general fibre of the rational map $\psi$, and let $\widetilde{S}$ be its proper transform on the threefold $\widetilde{X}$.
If $n\not\in\{122,127,129,130\}$, then $S$ is a hypersurface in the weighted projective space $\mathbb{P}(a_1,a_2,a_3,a_4)$ of degree $d$.
Similarly, if $n\in\{122,127,129,130\}$, then the surface $S$ is a complete intersection in $\mathbb{P}(a_0,a_1,a_2,a_3,a_4)$ 
of  two hypersurfaces: the hypersurface $X$ of degree $d$, and the hypersurface of degree $a_0a_1$ that is given by
$$
y^{a_0}=\lambda x^{a_1}
$$ 
for some $\lambda\in\mathbb{C}$, where $x$ and $y$ are coordinates on $\mathbb{P}(a_0,a_1,a_2,a_3,a_4)$ of weights $a_0$ and~$a_1$, respectively.
Thus, in all cases, the Kodaira dimension of the surfaces $S$ is negative by adjunction formula,
so that $\widetilde{S}$ is uniruled.
In fact, one can show that $\widetilde{S}$ is rational.

Now we can apply (relative) Minimal Model Program to $\widetilde{X}$ over $\mathbb{P}^1$.
One possibility is that we obtain a commutative diagram 
$$
\xymatrix{
\widetilde{X}\ar@{->}[d]_{\phi}\ar@{-->}[rr]^{\chi}&&V\ar@{->}[d]^{\eta}\\%
\mathbb{P}^1\ar@{=}[rr]&&\mathbb{P}^1}
$$
where $\chi$ is a birational map, $V$ is a threefold with terminal $\mathbb{Q}$-factorial singularities, 
$\mathrm{rk}\,\mathrm{Pic}(V)=2$, and $\eta$ is a fibration into del Pezzo surfaces.
Another possibility is that we obtain a commutative diagram
$$
\xymatrix{
\widetilde{X}\ar@{->}[d]_{\phi}\ar@{-->}[rr]^{\chi}&&V\ar@{->}[d]^{\eta}\\%
\mathbb{P}^1&&S\ar@{->}[ll]^{\upsilon}}
$$
where $\chi$ is a birational map, $V$ is a threefold with terminal $\mathbb{Q}$-factorial singularities,
$S$~is a normal surface, $\eta$ is a conic bundle, $\mathrm{rk}\,\mathrm{Pic}(V)=\mathrm{rk}\,\mathrm{Pic}(S)+1$,
and $\upsilon$ is a surjective morphism with connected fibres.
In fact, the surface $S$ has Du Val singularities \cite{Mori-Prokhorov}.

\begin{corollary}
\label{corollary:not-solid}
Let $X$ be a quasi-smooth Fano threefold in family \textnumero\,$n$ in Table~\ref{table-rat} such that $n\not\in\{100,101,102,103,110\}$.
Then $X$ is not solid.
\end{corollary}

In particular, if $n\not\in\{100,101,102,103,110\}$, then $X$ is not birationally rigid.

\section{Birationally non-rigid Fano threefolds}\label{bir-maps}
\label{Fano-models}

In this section, we will complete the proof of Theorem~\ref{main-theorem}. In each case, we regard the threefold as a hypersurface $X$ inside the weighted projective ambient space $\PP=\PP(a_0,a_1,a_2,a_3,a_4)$, and distinguish a quotient singularity $p\in X$; for families \textnumero\,$100$, \textnumero\,$101$, \textnumero\,$102$, \textnumero\,$103$ this point is $p=p_3=(0:0:0:1:0)$ and for family \textnumero\,$110$ it is $p=p_4=(0:0:0:0:1)$. The extremal extraction from this point is prescribed by Kawamata \cite{Kawamata-div} in a local analytic neighbourhood of $p$. We proceed by identifying a projective toric variety $T$ with Picard rank $2$ together with a birational morphism $\Phi:T\to\PP$. We then demand that $\varphi:Y\to X$ viewed locally near $p$ is the Kawamata blow up at the point $p$, where $Y=\Phi_*^{-1}(X)$ and $\varphi$ is simply the restriction of $\Phi$ to $Y$. Now, by construction, $Y$ has Picard number two, and hence admits a 2-ray game. We proceed by running the 2-ray game on $T$ and restricting it to $Y$, and each time we recover the 2-ray game of $Y$, which ends with a divisorial contraction to a Fano threefold not isomorphic to $X$. 
See \cite{AZ} for a comprehensive explanation of explicit 2-ray games and the links obtained from them.

We treat the first four families together, and then we consider family \textnumero\,$110$ separately as it has a higher index and behaves slightly differently. Let $X\subset\PP$ be a quasi-smooth member in one of the four families above with Fano index 2, and denote the variables by $x_0,\dots,x_4$. Note that in each case the defining polynomial $f(x_0,\cdots,x_4)$ contains two monomial $x_4^2$ and $x_3^3x_k$ with nonzero coefficient, where $k=2$ in families \textnumero\,$100$, \textnumero\,$102$, \textnumero\,$103$, and $k=0$ in family \textnumero\,$101$. This shows that $p$ is locally described by the quotient singularity
\[\frac{1}{a_3}(a_0,a_1,a_4-a_3)\text{ in the first three cases, and }\frac{1}{a_3}(a_1,a_2,a_4-a_3)\text{ in family \textnumero\,$101$}.\]

In each case $a_0+(a_4-a_3)\equiv0$ or $a_2+(a_4-a_3)\equiv0$ $\mod a_3$, and the third local coordinate weight is $2$, the index of $X$. In order to consider the Kawamata blow up of a 3-fold quotient singularity, the local description is expected to be of type
\[\frac{1}{r}(1,a,b)\text{, where }a+b\equiv 0\mod r.\]
For this, we can multiply the local weights above by $\lceil\frac{a_3}{2}\rceil$. If we denote the local coordinates of $\frac{1}{r}(1,a,b)$ with $x,y,z$, then the Kawamata blow up, with the new variable $u$, is given by 
\[u^\frac{1}{r}x,u^\frac{a}{r}y,u^\frac{b}{r}z.\]
This is because if the germ $\frac{1}{r}(1,a,b)$ is viewed as an affine toric variety with rays $\rho_x$, $\rho_y$ and $\rho_w$, then the Kawamata blow up of this toric variety is realised by adding a primitive new ray $\rho_u$ with relation
\[r\cdot\rho_u=1\cdot\rho_x+a\cdot\rho_y+b\cdot\rho_z.\]

We take advantage of this. We view $\PP$ as a projective toric variety defined by $5$ primitive rays in $\mathbb{Z}^4$ and one relation amongst them, namely
\[\sum_{i=0}^4 a_i\cdot\rho_i=0,\]
where $\rho_i$ is the primitive ray corresponding to the toric divisor $D_i=\{x_i=0\}$. Using this, we define a rank two toric variety $T$ with coordinate weights given by the matrix
\begin{equation}\label{matrixPt100}\left(\begin{array}{cccccc}
u&y_3&y_4&y_2&y_1&y_0\\
0&a_3&a_4&a_2&a_1&a_0\\
-a_3&0&b_4&b_2&b_1&b_0
\end{array}\right)\end{equation}
where $b_i$ is the smallest positive integer congruent to $\lceil\frac{a_3}{2}\rceil\cdot a_i$ modulo $a_3$ when $x_i$ appears as a local coordinate of the singular point $p\in X$, otherwise $b_i=a_4-a_3$. Let us explain this by an example. Suppose $X$ is in family \textnumero\,$100$. Then the defining equation of $X$ is of the form
\[x_3^3x_2=x_4^2+x_1^9+x_2^3x_4+\cdots .\]
In an analytic neighbourhood of $p$ we can eliminate the (tangent) variable $x_2$, which gives local coordinates $x_0,x_1,x_4$ with a finite group action of $\mathbb{Z}_5$ by
\[\frac{1}{a_3}(a_0,a_1,a_4-a_3)=\frac{1}{5}(1,2,4),\]
which can be seen as
\[\frac{1}{a_3}(b_0,b_1,b_4)=\frac{1}{5}(3,1,2).\]
Clearly, the Kawamata blow up is locally given by
\[(u,y_0,y_1,y_4)\mapsto(u^\frac{3}{5}y_0,u^\frac{1}{5}y_1,u^\frac{2}{5}y_4)=(x_0,x_1,x_4).\]
Plugging these into the equation of $X$ above shows that the multiplicity of $u$ in the right hand side of the equation is exactly $\frac{4}{5}$, precisely due to the presence of $x_4^2$ in the equation. Other cases are similar.

The map $\Phi:T\to \PP$ is defined by the restriction to the $y_3$-wall, which defines
\[(u,y_3,y_4,y_2,y_1,y_0)\mapsto(u^\frac{b_0}{a_3}y_0:u^\frac{b_1}{a_3}y_1:u^\frac{b_2}{a_3}y_2:y_3:u^\frac{b_4}{a_3}y_4)=(x_0:x_1:x_2:x_3:x_4).\]
Consequently, $Y\subset T$ is defined by the vanishing of the polynomial
\[
g(u,y_0,\cdots,y_4)=\frac{1}{u^\frac{a_3}{a_4-a_3}}\cdot f(u^\frac{b_0}{a_3}x_0,u^\frac{b_1}{a_3}x_1,u^\frac{b_2}{a_3}x_2,x_3,u^\frac{b_4}{a_3}x_4).
\]

The nonzero determinants of the $2\times 2$ minors of the matrix above are all divisible by $a_3$, which makes this particular description of $T$ ``not well-formed'' (see \cite{Hamid-Crelle} for a general treatment). However, this can easily be fixed by an action of the matrix
\[\left(\begin{array}{cc}
1/a_3&-2/a_3\\
0&1
\end{array}\right),\]
which, for example, transforms (\ref{matrixPt100}) in the case of family \textnumero\,$100$ into 
\[\left(\begin{array}{cccccc}
u&y_3&y_4&y_1&y_2&y_0\\
2&1&1&0&-1&-1\\
-5&0&2&1&4&3
\end{array}\right)\]
Note that we have also rearranged the columns of the matrix (easily tractable using the variables order). This is because we will run the 2-ray game on $T$ according to the GIT chambers of the action of $(\mathbb{C}^*)^2$.

Restricting the wall-crossing to $Y$ over the $y_4$-wall is an isomorphism on $Y$ as $g$ includes the term $y_4^2$. Crossing the next wall on $T$ is locally the flip $(2,1,1,-1,-1)$, read-off from the first row of the wellformed (\ref{matrixPt100}) after setting $y_i=1$, the variable corresponding to the next wall ($i=1$ for families \textnumero\,$100, 101, 102$ and $i=0$ for family \textnumero\,$103$). This restricts to an Atiyah flop $(1,1,-1,-1)$ as $g$ includes the monomial $uy_i^\alpha$ with nonzero coefficient, where $\alpha=\frac{\deg f}{2}$. Note that this coefficient is nonzero as a consequence of $X$ being quasi-smooth. The last wall, that is the $y_2$-wall, contracts the divisor $\{x=0\}$ on $T$ (or $\{x=0\}$ in the case of \textnumero\,$103$). The restriction to $Y$ is always a divisorial contraction to a Fano 3-fold with a singular point at the image of the contraction. Table \ref{modelsind2} contains the information on the image in each case.

\begin{table}[h]\small
\caption{\footnotesize Birational models for Fano hypersurfaces in families \text{\textnumero}\,$100, 101, 102, 103$.}\label{modelsind2}
$$
\renewcommand\arraystretch{1.5}
\begin{array}{|c|c|c|}
\hline
\text{\textnumero}&p_3\in X_d\subset\PP(a_0,a_2,a_2,a_3,a_4)&\text{New Model}\\
\hline
\hline
100 &\frac{1}{5}(3,1,2)\in X_{18}\subset\PP(1,2,3,5,9)&cE_6\in Z_{10}\subset\PP(1,1,1,3,5)\\
\hline
101 &\frac{1}{7}(1,5,2)\in X_{22}\subset\PP(1,2,3,7,11)&cE_7\in Z_{12}\subset\PP(1,1,1,4,6)\\
\hline
102 &\frac{1}{7}(4,1,3)\in X_{26}\subset\PP(1,2,5,7,13)&\frac{1}{2}(1,1,1,0;0)\in Z_{14}\subset\PP(1,1,2,4,7)\\
\hline
103 &\frac{1}{11}(1,7,4)\in X_{38}\subset\PP(2,3,5,11,19)&cE_8\in Z_{22}\subset\PP(1,1,3,7,11)\\
\hline
\end{array}
$$
\end{table}

Now we turn our attention to Family \textnumero\,$110$. A quasi-smooth member in this family is hypersurface $X_{21}\subset\PP(1,3,5,7,8)$ and contains a singular point of type $\frac{1}{8}(1,3,7)$ as the defining polynomial contain the monomial $x_4^2x_2$. In Kawamata format this singularity is of type $\frac{1}{8}(3,2,5)$. Similar to the computations above, the toric variety $T$ is given by the matrix

\[\left(\begin{array}{cccccc}
u&y_4&y_1&y_3&y_2&y_0\\
0&8&3&7&5&1\\
-8&0&1&5&7&3
\end{array}\right),\]
which can be wellformed to
\[\left(\begin{array}{cccccc}
u&y_4&y_1&y_3&y_2&y_0\\
5&7&2&3&0&-1\\
-8&0&1&5&7&3
\end{array}\right),\]
The 2-ray game of $T$ restricts to $Y$ via an isomorphism followed by a flip of type $(5,1,-3,-2)$, then it ends with a divisorial contraction to a singular point of type $cE_7$ on a hypersurface $Z_7\subset\PP(1,1,1,2,3)$.

\section{Evidence for Conjecture~\ref{conjecture:main}} 
\label{sketch}

Let us first prove that no smooth point or curve on any quasi-smooth member in families \textnumero\,$100, 101, 102, 103, 110$ can be a centre of maximal singularities (see \cite{Pukh-Max} for an introduction to the theory of maximal singularities).

\begin{lemma}\label{lemma-excluding-smooth-point}
Let $X$ be a quasi-smooth Fano hypersurface of degree $d$ in the weighted projective space $\mathbb{P}(a_0, a_1, a_2, a_3, a_4)$ that belongs to one of the families \textnumero\,$100, 101, 102,103,$ or $110$. Let $\mathcal{M}$ be a mobile linear subsystem in $|-nK_X|$ for some positive integer $n$. Then any smooth point in $X$ cannot be a centre of non-canonical singularities of the pair $(X,\frac{1}{n}\mathcal{M})$.
\end{lemma}
\begin{proof}
Let $p$ be a smooth point of $X$ and suppose that $p$ is a centre of non-canonical singularities of the pair $(X,\frac{1}{n}\mathcal{M})$. We then obtain
\[\mathrm{mult}_p(\mathcal{M}^2)>4n^2\]
from \cite[Corollary~3.4]{Corti}.
We now seek for a contradiction case by case.
For the families \textnumero\,$100, 101, 102,103$ we consider a general member $H$ in $|\mathcal{O}_X(a_1a_2a_3)|$ that passes through the point~$p$. Since the linear subsystem of  $|\mathcal{O}_X(a_1a_2a_3)|$ consisting of members passing through the point $p$ has a zero-dimensional base locus, $H$ contains no $1$-dimensional components of the base locus of $\mathcal{M}$. Then
\[\frac{4a_1a_2a_3n^2d}{a_0a_1a_2a_3a_4}=H\cdot\mathcal{M}^2\geqslant \mathrm{mult}_p(H)
\mathrm{mult}_p(\mathcal{M}^2)>4n^2,\]
which yields an absurd inequality $2>a_0$.

For the family \textnumero\,$110$ we consider a general member $H$ in $|\mathcal{O}_X(15)|$ that passes through the point $p$.   It is easy (but a bit tedious)  to check that  the base locus of the linear subsystem of~$|\mathcal{O}_X(15)|$ consisting of members passing through the point $p$ is zero-dimensional. Then
\[\frac{27 n^2}{8}=H\cdot\mathcal{M}^2\geqslant \mathrm{mult}_p(H)
\mathrm{mult}_p(\mathcal{M}^2)>4n^2,\]
which is absurd.
\end{proof}

\begin{lemma}\label{lemma-excluding-curve}
Under the same condition as Lemma~\ref{lemma-excluding-smooth-point}, any curve contained in the smooth locus of $X$ cannot be a centre of non-canonical singularities of the pair $(X,\frac{1}{n}\mathcal{M})$.
\end{lemma}
\begin{proof}
Let $C$ be a curve  that is contained in the smooth locus of $X$ and suppose that $C$ is a centre of non-canonical singularities of the pair $(X,\frac{1}{n}\mathcal{M})$. We then obtain
\[\mathrm{mult}_C(\mathcal{M})>n.\]
Choose two general members $H_1$ and $H_2$ from the mobile linear system $\mathcal{M}$. We then obtain
\[n^2\left(-K_X\right)^3=-K_X\cdot H_1\cdot H_2\geqslant \left( \mathrm{mult}_C(\mathcal{M})\right)^2\left(-K_X\cdot C\right)>n^2\left(-K_X\cdot C\right).\]
Since the curve $C$ is contained in the smooth locus of $X$, the inequality above implies that $$\left(-K_X\right)^3>1.$$
However, the anticanonical classes of the families \textnumero\,$100, 101, 102,103, 110$ have self-intersection numbers less than $1$. This completes the proof.
\end{proof}

Together with \cite[Theorem~5]{Kawamata-div}, the lemma above shows that a curve on $X$ cannot be a centre of non-canonical singularities of the pair $(X,\frac{1}{n}\mathcal{M})$. In conclusion, it follows from Lemmas~\ref{lemma-excluding-smooth-point}
and~\ref{lemma-excluding-curve} that only singular points of $X$ can be  centres of non-canonical singularities of the pair $(X,\frac{1}{n}\mathcal{M})$.

It remains to study the singular points on $X$. Recall from Section\,\ref{bir-maps} that in each family there exists a birational map to another Fano threefold constructed via an elementary Sarkisov link starting from the Kawamata blow up of the highest index quotient singularity. We now concentrate on lower index singularities and show that in each case of Fano index 2 these singular points produce no elementary link, and in the index 3 case we obtain a new birational map to a Fano threefold. Let us proceed with the latter.

A quasi-smooth $X_{21}\subset\PP(1,3,5,7,8)$ contains two singular points of type $\frac{1}{8}(1,3,7)$ and $\frac{1}{5}(3,2,3)$. A typical equation for this hypersurface is
\[f=x_2^4x_0+x_4^2x_2+x_3^3+x_1^7+\cdots .\]
Indeed, the tangent space to $p=p_2$ is determined by $x_0$ and the local coordinates of the tangent space are $x_1,x_3,x_4$. In Kawamata format this singularity is of type $\frac{1}{5}(1,4,1)$ and the local blow up is given by
\[(u,y_1,y_3,y_4)\mapsto (u^\frac{1}{5}x_1,u^\frac{4}{5}x_3,u^\frac{1}{5}x_4).\]
So we can define the rank two toric variety $T$ via the matrix
\[\left(\begin{array}{cccccc}
u&y_2&y_4&y_1&y_3&y_0\\
0&5&8&3&7&1\\
-5&0&1&1&4&2
\end{array}\right)\]
together with the map
\[\Phi(u,y_2,y_4,y_1,y_3,y_0)=(u^\frac{2}{5}x_0:u^\frac{1}{5}x_1:x_2:u^\frac{4}{5}x_3:u^\frac{1}{5}x_4)\in\PP(1,3,5,7,8).\]
Note that the blow up weight of the variable $x_0$ is determined by the multiplicity of $f$, which is $\frac{2}{5}$.
It is easy to see that $Y$ does not follow the 2-ray game of $T$. This is because $g$, the defining equation of $Y$, belong to the irrelevant ideal $(u,y_2)\cap(y_4,y_1,y_3,y_0)$ as
\[g=u(y_1^7+uy_3^3+\cdots)+y_2(y_4^2+y_2^3y_0+\cdots).\]
This can be resolved by replacing $T$ with a higher dimensional toric variety (see \cite{AZ} for an explanation and examples), using an unprojection defined by the cross-ratio 
\[y=\frac{y_1^7+uy_3^3+\cdots}{-y_2}=\frac{y_4^2+y_2^3y_0+\cdots}{u},\]
which defines an isomorphism between $Y$ and the complete intersection of
\[yy_2+y_1^7+uy_3^3+\cdots=-uy+y_4^2+y_2^3y_0+\cdots=0\]
inside the toric variety defined by
\[\left(\begin{array}{ccccccc}
u&y_2&y_4&y_1&y&y_3&y_0\\
0&5&8&3&16&7&1\\
-5&0&1&1&7&4&2
\end{array}\right)\]
which can be wellformed into
\[\left(\begin{array}{ccccccc}
u&y_2&y_4&y_1&y&y_3&y_0\\
3&1&1&0&-1&-1&-1\\
21&7&8&1&0&-3&-5
\end{array}\right).\]
The 2-ray game of the toric variety restricts to a 2-ray game on $Y$ by two isomorphisms followed by a flip of type $(8,1,-3,-5)$ and ends in a divisorial contraction into a $cE_7/2$ singular point on the complete intersection $Z_{6,7}\subset\PP(1,1,2,2,3,5)$.

The exclusion in all other quotient singularities for the 4 families of index 2 follow the same principle. We first identify the quotient singularity, with a ``key monomial'', which defines the tangent space. Then we construct the rank two toric variety with weights prescribed by Kawamata blow up and one weight dictated by the tangent multiplicity. If needed, we perform an unprojection as above. In each case we either obtain no link, when $-K_Y$ is outside the closure of the cone of movable divisors, or we obtain a bad link, when $-K_Y$ is in the boundary of the cone of movable divisors (see \cite{AZ} for justification of exclusions). So, neither case results in an elementary Sarkisov link. We capture these in Table\,\ref{table-exclusions}, with the essential data needed to carry on the computation. Instead of writing the full matrix of $T$ in each case we only write the blow up weights indexed by the corresponding variable in anticlockwise order of the GIT chambers. Whenever an unprojection is needed there is an extra variable $y$ in the list of weight-variables denoted by $m_y(n)$, where $m$ is weight corresponding to the blow up and $n$ is the other weight (in relation to the weighted projective space $\PP$).

\begin{table}[h]\small
\caption{\footnotesize Exclusion of links in families \text{\textnumero}\,$100, 101, 102, 103$.}\label{table-exclusions}
$$
\renewcommand\arraystretch{1.5}
\begin{array}{|c|c|c|c|c|}
\hline
\text{\textnumero}&\text{singularity}&\text{key monomial}&\text{blow up}&\text{exclusion}\\
\hline
\hline
100 &2\times\frac{1}{3}(1,2,2)&x_4^2+x_4x_2^2&-3_u,0_{y_2},1_{y_3},3_{y_4},6_y(15),1_{y_1},2_{y_2}&\text{bad link}\\
\hline
101 &\frac{1}{3}(1,2,2)&x_2^3x_4\text{ or }x_2^7x_0&-3_u,0_{y_2},1_{y_4},2_{y_4},1_{y_1},2_{y_0}&\text{bad link}\\
\hline
102 &\frac{1}{5}(2,2,3)&x_2^5x_0&-5_u,0_{y_2},1_{y_3},4_{y_4},8_y(21),1_{y_1},3_{y_0}&\text{bad link}\\
\hline
103 &\frac{1}{3}(1,1,1)&\begin{array}{c}
x_1^9x_4\\
x_1^{11}x_2\\
x_1^{13}x_0
 \end{array}&\begin{array}{c}
 -3_u,0_{y_1},2_{y_4},1_{y_2},4_{y_3},1_{y_0}\\
-3_u,0_{y_1},1_{y_3},2_{y_4},1_{y_0},4_{y_2} \\
 -3_u,0_{y_1},1_{y_3},2_{y_4},1_{y_2},4_{y_0}
 \end{array}&\begin{array}{c}
 \text{no link}\\
 \text{bad link}\\
 \text{no link}
 \end{array}\\
\hline
 103&\frac{1}{5}(2,1,4)&x_2^7x_1&-5_u,0_{y_2},2_{y_4},3_{y_3},1_{y_0},4_{y_1}&\text{bad link}\\
\hline
\end{array}
$$
\end{table}

\bibliographystyle{amsplain}
\bibliography{bib}

\end{document}